\documentclass{article}
\usepackage{amsmath,amsthm,amssymb}
\usepackage{cite}

\newtheorem{theorem}{Theorem}[section]
\newtheorem{corollary}[theorem]{Corollary}
\newtheorem{proposition}[theorem]{Proposition}

\theoremstyle{remark}
\newtheorem{remark}[theorem]{Remark}
\newtheorem{definition}[theorem]{Definition}

\numberwithin{equation}{section}
\allowdisplaybreaks[3]

\newcommand\1{1\!\!1}
\newcommand\BBg{\B_{\mathrm{b}}(\Ga_0^2)}
\newcommand\BBs{B_{\mathrm{bs}}(\Ga_0^2)}
\newcommand\Bbg{\B_{\mathrm{b}}(\Ga_0)}
\newcommand\Bbs{B_{\mathrm{bs}}(\Ga_0)}
\newcommand\Bb{\B_{\mathrm{c}}({X})}
\newcommand\B{\mathcal{B}}
\newcommand\D{\mathcal{D}}
\newcommand\cnt[2]{\text{\setbox2=\hbox{#1}\rlap{\hbox to \wd2{\hfil#2\hfil}}\box2}}
\newcommand\eps{\varepsilon}

\newcommand\Ga{\Gamma}

\newcommand\I{{\mathcal I}}
\newcommand\K{\mathcal{K}}
\newcommand\La{\Lambda}
\newcommand\la{\lambda}
\newcommand\ls{\mathrm{ls}}
\newcommand\MLf{\M_\mathrm{lf}(\Ga_0^2)}
\newcommand\Mlf{\M_\mathrm{lf}(\Ga_0)}
\newcommand\M{\mathcal{M}}
\newcommand\N{\mathbb{N}}
\newcommand\Ob{\mathcal{O}_{\mathrm{c}}({X})}
\newcommand\R{\mathbb{R}}
\newcommand\Star{\mathbin{\cnt{$\bigcirc$}{$\star$}}}

\DeclareMathOperator*{\esssup}{ess\,sup}

\author{Dmitri Finkelshtein\thanks{Institute of Mathematics,
         National Academy of Sciences of Ukraine,
         01601 Kiev-4, Ukraine, e-mail:fdl@imath.kiev.ua}}
\title{Towards on convolutions on configuration spaces. I.~Spaces of finite configurations}

\begin{document}

\maketitle
\begin{abstract}
We consider two types of convolutions ($\ast$ and $\star$) of functions on spaces of finite configurations (finite subsets of a phase space), and some their properties are studied. A connection of the $\ast$-convolution with the convolution of measures on spaces of finite configurations is shown. Properties of multiplication and derivative operators with respect to the $\ast$-convolution are discovered. We present also conditions when the $\ast$-convolution will be positive definite with respect to the $\star$-convolution.
\end{abstract}



\section{Introduction}

Spaces of configurations (discrete subsets of a phase space) have been became a separate mathematical object of investigation starting from 1960'th. They were studied in different branches of mathematics such as functional analysis, mathematical physics, probability theory, topology. Finite and locally finite subsets of a phase space are useful objects for describing of mathematical models for different systems in applications: in physics, chemistry, biology, economics, social sciences etc. The corresponding interpretations for elements of the subsets are molecules, individuals, agents etc. In turn, the phase space may be a discrete set, for instance, a lattice or, more generally, a graph in, say, an Euclidean space. Lattice systems were intensively studied in literature, see e.g. \cite{Lig1985,Lig1999,KL1999,MP1991,Pre2009}. In the case then the phase space is a continuum set, for instance, an Euclidean space or, more generally, a topological space (say, a manifold), the systems describing by the corresponding space of configurations are called continuous.

In number of problems of statistical physics a physical system may be modeled by huge or even infinite set in a continuum phase space. Mathematical description of such systems was initiated in XIX~century by L.\,Boltzmann and his followers, see e.g. \cite{Cer1988,BE1983}. In~XX~century this area was intensively studied, started from the fundamental papers by J.\,W.\,Gibbs, see e.g. \cite{Gib1960}, what were background for the modern theory of Gibbs measures on configuration spaces. Started from 1940th mathematical models of continuous systems in statistical physics were actively studied by N.\,N\,Bogolyubov, see e.g. \cite{Bog1962}, and his followers. In 1960th the necessity of a rigorous description for spaces of locally finite configuration and states (probability measures) on such spaces became clear. The corresponding investigations were started by R.\,L.\,Dobrushin, O.\,Lenford, and D.\,Ruelle, see e.g. the review \cite{DSS1989} and the references therein. Meanwhile, the detailed analysis on spaces of configurations was initiated in paper \cite{VGG1975} by A.\,M.\,Vershik, I.\,M.\,Gelfand, M.\,I.\,Graev. The modern form of the analysis on spaces of configurations was gained in papers by S.\,Albeverio, Yu.\,Kondratiev, M.\,R\"{o}ckner and their followers, see e.g. \cite{AKR1998a,AKR1998,Roc1998,Kun1999}.

An actual bibliography for papers about configuration spaces over a continuum phase space may become now a scientific paper itself. Even just list of authors need to much journal's space. Therefore, we restrict ourselves to enumerate the main areas of this business that have a long history and are intensively developed now. Namely, study of topological, metric, measurable, and algebraic structures on spaces of configurations; measure theory, in particular, study of Gibbs and Cox measures, determinantal and permanent measures; calculus, differential geometry, and harmonic analysis on configuration spaces; deterministic and stochastic dynamical systems, their ergodic and invariant measures; Markov evolutions, equilibrium and non-equilibrium stochastic dynamics, in particular, diffusion, birth-and-death, jump, Hamiltonian dynamics; different scalings of the dynamics above, hydrodynamic and kinetic equations etc.

The present paper deals with different convolutions (between functions and between measures) on spaces of configurations which were studied in literature. These questions are important for the development of a rich analysis on configuration spaces. On the other hand the convolutions are actively used for further investigations, in particular, for study of stochastic dynamics on configuration spaces. Due to the journal limitation the publication is divided on two parts. This first part is devoted to the convolutions on the spaces of finite configurations, that are discrete subsets of a continuum phase space which have an arbitrary but a finite number of points. It is worth noting that spaces of finite configurations are a object of infinite-dimensional analysis. Note also that spaces of locally finite subsets (which are considered in the second paper) are not a direct generalization of spaces of finite configurations. More exact description of their connection is an analogy with the correspondence between Hilbert spaces and the sequence space $l_2$ that may be considered as a space of Fourier coefficients. This approach to harmonic analysis on spaces of locally finite subsets was initiated in~\cite{KK2002} and was applied in dozens of recent publications.

The present paper is organized as follows. In Section 2 we describe the main structures and their properties on spaces of finite configurations which will be used in both papers. In Section 3 we study the so-called Ruelle $*$-calculus, defined via the Ruelle convolution~\cite{Rue1964}. This $*$-convolution was used often in literature, however, some its analytic properties were unknown before. In Section 4 we consider a multiplication operator with respect to the Ruelle convolution. In Section 5 there are several subjects, namely: a convolution of measures on spaces of finite configurations and its connection with the Ruelle convolution of functions, properties of the generating functional (the so-called Bogolyubov functional, see details in e.g. \cite{KKO2006}), properties of the derivative operator with respect to the Ruelle convolution, the connection between the $*$-convolution and the $\star$-convolution introduces in~\cite{KK2002} by Yu.\,Kondratiev and T.\,Kuna.

Author would like to thank Prof. Dr. Yuri Kondratiev for useful discussions. The paper was partially supported by The Ukraine President Scholarship and Grant for young scientists.

\section{Spaces of finite configurations}
Let $X$ be a connected oriented non-compact Riemannian $C^\infty$-manifold, $\mathcal{O}({X})$ be a class of all open subsets from ${X}$, $\B({X})$ be the corresponding Borel $\sigma$-algebra. We denote classes of all open and Borel subsets from ${X}$ which have compact supports by $\Ob$ and~$\Bb$, correspondingly. Let $m$ be a non-atomic Radon measure on ${X}$, i.e., $m(\La)<\infty$, $\La\in\Bb$ and $m(\{x\})=0$, $x\in{X}$. Suppose also that there exists a sequence $\{\La_n\}_{n\in\N}\subset\Bb$ such that $\La_n\subset\La_{n+1}$, $n\in\N$ and $\bigcup_{n\in\N}\La_n=X$.

For any $Y\in\B({X})$ and~$n\in \N_0:=\N\cup \{0\}$, the set
 \begin{equation*}
 \Ga_Y^{(n)}:=\bigl\{ \eta \subset
 Y\bigm| |\eta |=n\bigr\}, \ n\in\N; \qquad \Ga_Y^{(0)}:=\{\emptyset\}
 \end{equation*}
is said to be a space of all $n$-point configurations on the set $Y$. Here and subsequently, the symbol $|\cdot|$ means a number of points in a discrete set.
For any $\La\in\Bb$, $\La\subset Y$, we denote $\eta_\La:=\eta\cap\La$ and consider a mapping $N_\La:\Ga_{0,Y}^{(n)}\to\N_0$, given by
 $N_\La(\eta):=|\eta_\La|$. For $n\in\N$, one set
 \[
 \widetilde{Y^n} = \bigl\{ (x_1,\ldots ,x_n)\in Y^n\bigm|
 x_k\neq x_l, \text{ if only } k\neq l\bigr\} .
 \]
We consider also a mapping
$\mathrm{sym}_{Y,n}:\widetilde{Y^n}\to\Ga_Y^{(n)}$,
$\mathrm{sym}_{Y,n}\bigl( (x_1,\ldots ,x_n)\bigr) := \{x_1,\ldots
,x_n\}$.
Then, one can identify the space of all $n$-point configurations $\Ga_Y^{(n)}$ with the quotient of $\widetilde{Y^n}$ with respect to the natural action of the permutation group $S_n$ on $\widetilde{Y^n}$. Thus, one can define in the space $\Ga_Y^{(n)}$ the family of open sets
 $\mathcal{O}\bigl(\Ga_{0,Y}^{(n)}\bigr):=\mathrm{sym}_{Y,n}^{-1}\bigl(
\mathcal{O}(\widetilde{Y^n})\bigr)$. The base of topology is formed by the system of sets
\[
 U_1\widehat{\times}\cdots\widehat{\times}U_n:=
 \bigl\{\eta\in\Ga_{0,Y}^{(n)}\bigm|N_{U_1}(\eta)=1,\ldots,
 N_{U_n}(\eta)=1\bigr\},
\]
where $U_1,\ldots,U_n\in\mathcal{O}_c({X})$, $U_1,\ldots,U_n\subset Y$ and
$U_i\cap U_j=\emptyset$ if only $i\neq j$. The Borel $\sigma$-algebra $\B\bigl(\Ga_Y^{(n)}\bigr)$ corresponding to $\mathcal{O}\bigl(\Ga_{0,Y}^{(n)}\bigr)$ coincides with the $\sigma$-algebra given by the family of mappings $N_\La$, $\La \in \Bb$, $\La\subset Y$, see e.g. \cite{AKR1998a}.

The space of finite configurations on a set $Y\in\B({X})$ is a disjoint union
 \begin{equation}\label{Ga0Y}
  \Ga_{0,Y}:=\bigsqcup_{n\in \N_0}\Ga_Y^{(n)}.
 \end{equation}
The structure of a disjoint union allows to define a topology
 $\mathcal{O}(\Ga_{0,Y})$ on $\Ga_{0,Y}$. The corresponding Borel
 $\sigma$-algebra we denote by $\B(\Ga_{0,Y})$.
 In the case then $Y={X}$ we will omit a subscript, i.e.,
 $\Ga^{(n)}:=\Ga_{X}^{(n)}$, $ \Ga_0:=\Ga_{0,{X}}$.

A set $B\in\B(\Ga_0)$ is said to be bounded if there exist $\La\in\Bb$ and~$N\in\N$, suh that $B\subset \bigsqcup_{n=0}^N\Ga_\La^{(n)}$. The class of all bounded sets from $\B(\Ga_0)$ we denote by $\Bbg$. A measure $\rho$ on~$\bigl(\Ga_0,\B(\Ga_0)\bigr)$ is said to be locally finite if $\rho(B)<\infty$ for all
$B\in\Bbg$. Let $\Mlf$ denote the class of all locally finite measures on~$\bigl(\Ga_0,\B(\Ga_0)\bigr)$. An important example of a locally finite measure on~$\Ga_0$ is the Lebesgue--Poisson measure that is defined as follows. The image on the space $\B(\Ga^{(n)})$ of the product-measure $m^{\otimes
n}$ on~$\widetilde{({X})^n}$ under the map
$\mathrm{sym}_{{X},n}$ we denote by $m^{(n)}$. The latter measure is well-defined since
$ m^{\otimes n}\bigl(({X})^n\setminus\widetilde{({X})^n}\bigr)=0$.
For $n=0$, we set $m^{(0)}(\{\emptyset\}):=1$.
Let $z>0$ be given. The Lebesgue--Poisson measure $\la_{z}$ on~$\bigl(\Ga_0,\B(\Ga_0)\bigr)$ is defined correspondingly to the expansion \eqref{Ga0Y} in the following way:
 \begin{equation}\label{la_z}
 \la_z :=\sum_{n=0}^\infty \frac {z^n}{n!} m^{(n)}.
 \end{equation}
For any $\La\in\Bb$, we preserve the same notation $\la_z$ for the restriction of $\la_z$ onto~$\Ga_{0,\La}$. The positive number $z$ is the intensity (the activity parameter) of the measure $\la_z$. For the case $z=1$, we will omit the subscript, namely, $\la:=\la_1$.
In \cite{Kun1999}, it was shown that, for any $A\in\B({X})$ with $m(A)=0$,
 \begin{equation*}
 \la\bigl(\{\eta\in\Ga_{0,Y} \mid \eta\cap
 A\neq\emptyset\}\bigr)=0,\qquad Y\in\B({X}).
 \end{equation*}
In particular, one can put $Y={X}$. Therefore, for any $\xi\in\Ga_0$, $x\in{X}$,
\begin{equation}\label{zeroset1}
    \la\bigl(\{\eta\in\Ga_0 \mid x\in\eta\}\bigr)
    =\la\bigl(\{\eta\in\Ga_0 \mid \xi\cap\eta\neq\emptyset\}\bigr)=0.
\end{equation}

Let us consider some classes of real-valued functions on~$\Ga_0$. In the sequel, a measurable function on~$\Ga_0$ will always mean a $\B(\Ga_0)/\B(\R)$-measurable function. Let $L^0(\Ga_0)$ denote the class of all measurable functions on~$\Ga_0$.
By the expansion \eqref{Ga0Y}, any function
$G\in L^0(\Ga_0)$ might be given by the system of its restrictions $G^{(n)}:=G\upharpoonright_{\Ga^{(n)}}$.
For a symmetric function
$G^{(n)}\circ\mathrm{sym}_{{X},n}^{-1}:\widetilde{({X})^n}\to\R$ we stand the same notation $G^{(n)}$ if this does not lead to misunderstanding.
A function~$G\in L^0(\Ga_0)$ is said to have a local support if there exists $\La\in\Bb$ such that $G\upharpoonright_{\Ga_0\setminus \Ga_{0,\La}}=0$.
Let $L_\ls^0(\Ga_0)$ denote the set of all measurable functions on $\Ga_0$ with local supports. Similarly, a function $G\in L^0(\Ga_0)$ is said to have a bounded support if there exists $B\in\Bbg$ such that
 $G\upharpoonright_{\Ga_0\setminus B}=0$.
Let $\Bbs$ denote the set of all bounded measurable functions on $\Ga_0$.

For any $\B({X})$-measurable function $f:{X}\to\R$ we define the Lebesgue--Poisson exponent as the function on $\Ga_0$, given by:
\begin{equation}\label{defexpLP}
e_\la(f,\eta):=\prod\limits_{x\in\eta}f(x),\quad
 \eta\in\Ga_0\setminus\{\emptyset\},\qquad
e_\la(f,\emptyset):=1.
\end{equation}

\section{$*$-calculus}

\begin{definition}[{see e.g. \cite{Rue1964}}]
For any $G_1,G_2\in L^0(\Ga_0)$, we define the following convolution rule on
$\Ga_0$:
\begin{equation}\label{ast}
 (G_1* G_2)(\eta):=\sum_{\xi\subset\eta}
 G_1(\xi)\,G_2(\eta\setminus\xi).
\end{equation}
\end{definition}
In particular, for any measurable $f,g:{X}\to\R$, one get
\begin{equation}
e_\lambda(f) * e_\lambda(g) = e_\lambda(f+g),  \label{addexp}
\end{equation}
taking into account \eqref{defexpLP} and \eqref{ast}.

\begin{proposition}[{see e.g. \cite{Kun1999}}]\label{Minlos-ast}
For any $H,G_1,G_2\in L^0(\Ga_0)$, the following identity holds true
\begin{equation}\label{minlosid-ast}
\int_{\Ga_0}H(\eta)(G_1\ast G_2)(\eta)d\la(\eta) =
\int_{\Ga_0}\int_{\Ga_{0}}H(\eta\cup\xi)G_1(\eta)G_2(\xi)d\la(\xi)d\la(\eta),
\end{equation}
if at least one of integrals is well-defined.
\end{proposition}

Let $C>0$ and~$\delta\geq 0$. We consider a Banach space
\begin{equation*}
\K_{C,\,\delta} = \bigl\{
k:\Ga_0\to\R \bigm| |k(\eta)|\leq \mathrm{const}\cdot C^{|\eta|}
(|\eta| !)^\delta \text{ for } \la\text{-a.a. } \eta\in\Ga_0 \bigr\}
\end{equation*}
with norm $\|k\|_{{C,\,\delta}}:=\esssup_{\eta\in\Ga_0}
\dfrac{|k(\eta)|}{C^{|\eta|} (|\eta| !)^\delta}$.
Clearly, for any $C'\geq C$, $\delta'\geq\delta$, the following inclusion holds
$\K_{C,\,\delta}\subset\K_{C',\,\delta'}$. For $\delta=0$, we will omit this subscript, namely, $\K_C:=\K_{C,0}$.

\begin{proposition}\label{convest}
Let $C_1, C_2>0$, $\delta_1, \delta_2\geq 0$, and
$k_i\in\K_{C_i,\,\delta_i}$, $i=1,2$. Then the function $k:=k_1\ast k_2$ belongs to the space $\K_{C,\,\delta}$, where~$C=C_1+C_2$,
$\delta=\max\{\delta_1,\,\delta_2\}$. Moreover, the Young-type inequality holds
\begin{equation}\label{Young1}
\|k_1\ast k_2\|_{C,\,\delta}\leq \|k_1\|_{C_1,\,\delta_1} \cdot
\|k_2\|_{C_2,\,\delta_2}.
\end{equation}
If $\delta\geq1$, $C_1\neq C_2$ then the function $k_1\ast k_2$ belongs to a more narrow
space $\K_{\bar{C},\,\delta}$, where
$\bar{C}=\max\{C_1,\,C_2\}$, and the corresponding inequality holds
\begin{equation}\label{Young2}
\|k_1\ast k_2\|_{\bar{C},\,\delta}\leq \frac{\bar{C}}{\bigl\vert C_1-C_2\bigr\vert}\|k_1\|_{C_1,\,\delta_1} \cdot
\|k_2\|_{C_2,\,\delta_2}.
\end{equation}
If $\delta\geq1$, $C_1=C_2$ then the function $k_1\ast k_2$ belongs to the space
$\K_{C',\,\delta}$ for any $C'>C_1$, moreover,
\begin{equation}\label{Young3}
\|k_1\ast k_2\|_{C',\,\delta}\leq \frac{C'}{eC_1\ln\frac{C'}{C_1}}\|k_1\|_{C_1,\,\delta_1} \cdot
\|k_2\|_{C_2,\,\delta_2}.
\end{equation}
If $k_1\in\K_{C_1,\delta_1}$, $C_1>1$, $\delta_1\geq 1$, and $k_2\in L^\infty(\Ga_0):=L^\infty(\Ga_0,d\la)$, then~$k_1*k_2\in \K_{C_1,\delta_1}$ and
\begin{equation}\label{Young4}
\|k_1\ast k_2\|_{C_1,\,\delta_1}\leq \frac{C_1}{C_1-1}\|k_1\|_{C_1,\,\delta_1} \cdot
\|k_2\|_{L^\infty(\Ga_0)}.
\end{equation}
Finally, if $k_1,k_2\in L^\infty(\Ga_0)$, then~$k_1*k_2\in\K_{C,0}$ for all $C\geq2$ and $k_1*k_2\in\K_{C,\delta}$ for all $\delta>0$, $C>0$, in particular,
\begin{equation}\label{Young5}
    \|k_1\ast k_2\|_{C,0}\leq \|k_1\|_{L^\infty(\Ga_0)}\|k_2\|_{L^\infty(\Ga_0)}, \quad C\geq2.
\end{equation}
\end{proposition}

\begin{proof}
For $\la$-a.a. $\eta\in\Ga_0$, one has
\begin{align}
&C^{-|\eta|}(|\eta|!)^{-\delta}|k(\eta)| \leq
C^{-|\eta|}(|\eta|!)^{-\delta}\sum_{\xi\subset\eta}
|k_1(\xi)| |k_2(\eta\setminus\xi)| \notag\\
&= C^{-|\eta|}(|\eta|!)^{-\delta}\sum_{\xi\subset\eta}
C_1^{|\xi|}(|\xi|!)^{\delta_1}
\frac{|k_1(\xi)|}{C_1^{|\xi|}(|\xi|!)^{\delta_1}}
\frac{|k_2(\xi)|}{C_2^{|\eta\setminus\xi|}(|\eta\setminus\xi|!)^{\delta_2}} C_2^{|\eta\setminus\xi|}(|\eta\setminus\xi|!)^{\delta_2}\notag\\
&\leq \|k_1\|_{C_1,\,\delta_1} \|k_2\|_{C_2,\,\delta_2} C^{-|\eta|}
\sum_{k=0}^{|\eta|}\frac{|\eta|!}{k! (|\eta|-k)!}
\frac{C_1^{k}(k!)^{\delta_1}}{(|\eta|!)^{\delta}}
C_2^{|\eta|-k}((|\eta|-k)!)^{\delta_2}\notag\\
&\leq \|k_1\|_{C_1,\,\delta_1} \|k_2\|_{C_2,\,\delta_2}C^{-|\eta|}
\sum_{k=0}^{|\eta|}\left(\frac{|\eta|!}{k!
(|\eta|-k)!}\right)^{1-\delta} C_1^{k}C_2^{|\eta|-k}\label{sepsa}
\\
&\leq \|k_1\|_{C_1,\,\delta_1} \|k_2\|_{C_2,\,\delta_2}C^{-|\eta|}
\sum_{k=0}^{|\eta|}\frac{|\eta|!}{k! (|\eta|-k)!} C_1^{k}C_2^{|\eta|-k}\notag \\
&=\|k_1\|_{C_1,\,\delta_1}\|k_2\|_{C_2,\,\delta_2},\notag
\end{align}
which completes the proof of the first statement.

Let now $\delta\geq1$. Then, by~\eqref{sepsa}, we derive
\begin{align*}
\bar{C}^{-|\eta|}(|\eta|!)^{-\delta}|k(\eta)| &\leq  \|k_1\|_{C_1,\,\delta_1} \|k_2\|_{C_2,\,\delta_2}C^{-|\eta|}
\sum_{k=0}^{|\eta|} C_1^{k}C_2^{|\eta|-k}\\&=\|k_1\|_{C_1,\,\delta_1} \|k_2\|_{C_2,\,\delta_2}\bar{C}^{-|\eta|} \frac{C_1^{|\eta|+1}-C_2^{|\eta|+1}}{C_1-C_2}.
\end{align*}
For definiteness, consider $\bar{C}=\max\{C_1,C_2\}=C_1$. Then
\[
\esssup_{\eta\in\Ga_0}C^{-|\eta|} \frac{C_1^{|\eta|+1}-C_2^{|\eta|+1}}{C_1-C_2}=
\esssup_{\eta\in\Ga_0}\frac{C_1-\bigl(\frac{C_2}{C_1}\bigr)^{|\eta|}C_2}{C_1-C_2}=\frac{C_1}{C_1-C_2}
\]
and the second statement is fulfilled.

In the case when~$C_1=C_2$, we have
\begin{align*}
(C')^{-|\eta|}(|\eta|!)^{-\delta}|k(\eta)| &\leq  \|k_1\|_{C_1,\,\delta_1} \|k_2\|_{C_2,\,\delta_2}(C')^{-|\eta|}
\sum_{k=0}^{|\eta|} C_1^{|\eta|}\\&=\|k_1\|_{C_1,\,\delta_1} \|k_2\|_{C_2,\,\delta_2} \Bigl(\frac{C_1}{C'}\Bigr)^{|\eta|}|(\eta|+1),
\end{align*}
and the result is followed by properties of the elementary function:
\[
\max_{x\geq1}(x+1)a^x=-\frac{1}{ae\ln a}, \quad a\in(0; 1).
\]

Let now $k_2\in L^\infty(\Ga_0)$. Then
\begin{align*}
&  C_1^{-|\eta|}(|\eta|!)^{-\delta_1}\sum_{\xi\subset\eta}
|k_1(\xi)| |k_2(\eta\setminus\xi)|\leq \|k_2\|_{L^\infty(\Ga_0)} \|k_1\|_{{C_1,\,\delta_1}} C_1^{-|\eta|}(|\eta|!)^{-\delta_1}\sum_{\xi\subset\eta}
C_1^{|\xi|}(|\xi|!)^{\delta_1}\\&=\|k_2\|_{L^\infty(\Ga_0)} \|k_1\|_{{C_1,\,\delta_1}} C_1^{-|\eta|}(|\eta|!)^{-\delta_1}\sum_{k=0}^{|\eta|}\frac{|\eta|!}{k!(|\eta|-k)!}
C_1^{k}(k!)^{\delta_1}\\&\leq\|k_2\|_{L^\infty(\Ga_0)} \|k_1\|_{{C_1,\,\delta_1}} C_1^{-|\eta|}\sum_{k=0}^{|\eta|}C_1^{k}=\|k_2\|_{L^\infty(\Ga_0)} \|k_1\|_{{C_1,\,\delta_1}} \frac{C_1-C_1^{-|\eta|}}{C_1-1},
\end{align*}
and for $C_1>1$ one has:
\[
\esssup_{\eta\in\Ga_0}\dfrac{C_1-C_1^{-|\eta|}}{C_1-1}=
\dfrac{C_1}{C_1-1},
\]
that proves \eqref{Young4}.

The last statement is followed by the equalities $\sum_{\xi\subset\eta}1=2^{|\eta|}$ and $\esssup_{\eta\in\Ga_0}\bigl(\frac{2}{C}\bigr)^{|\eta|}=1$ if only $C\geq2$.
\end{proof}

\begin{corollary}\label{cor1}
Let $k\in\K_{C,\delta}$, $C>0$, $\delta\geq0$. Then, for $\delta\in[0;1)$,
$k^{*n}\in\K_{nC,\delta}$, $n\in\N$ and $\|k^{*n}\|_{{nC,\delta}}\leq \|k\|_{{C,\delta}}^n$. In the case when $\delta\geq 1$, we have, for any $C'>C$,
$k^{*n}\in\K_{C',\delta}$, $n\geq2$, and
\[
\|k^{*n}\|_{{C',\delta}}\leq \biggl(\frac{C'}{C'-C}\biggr)^{n-2}\frac{C'}{eC\ln\frac{C'}{C}}\|k\|_{{C,\delta}}^n, \quad n\geq 2.
\]
Finally, if $k\in L^\infty(\Ga_0)$ then $k^{*n}\in\K_{C,0}$ for all $C\geq 2$, $n\in\N$, and
\[
\|k^{*n}\|_{{C,0}}\leq \biggl(\frac{C}{C-1}\biggr)^{n-2}\|k\|_{L^\infty(\Ga_0)}^n, \quad n\geq 2.
\]
\end{corollary}

The sequel results of this Section is somehow ``folk art''. They either are given in literature without proof like in \cite{Rue1964} or they can be derived from some informal general considerations like in \cite{She1971,She1972,She1973}. Hence, for the convenience of the reader, we present all these results with detailed proofs.

For an arbitrary $c\in\R$, we consider the set ${\mathcal{I}}_{c}$ of all measurable functions on~$\Ga_0$, such that~$u(\emptyset)=c$. Since $(u_1\ast
u_2)(\emptyset)=u_1(\emptyset)u_2(\emptyset)$, the set ${\mathcal{I}}_0$ is an ideal in the algebra $L^0(\Ga_0)$ with a product $\ast$. A unit in this algebra is the function
\[
u^{\ast 0}(\eta):=1^\ast(\eta):=0^{|\eta|}.
\]

For any $u\in L^0(\Ga_0)$ and $n\in{\N}$, one has
\begin{equation*}
u^{\ast n}(\eta)=(\underbrace{u\ast\ldots\ast u}_n)(\eta)=\sum_{\eta_1%
\sqcup\ldots\sqcup\eta_n=\eta}u(\eta_1)\ldots u(\eta_n),\quad \eta\in\Gamma_0,
\end{equation*}
therefore, for $u\in{\mathcal{I}}_{0}$, we get
\begin{equation*}
u^{\ast n} (\eta)=0,\quad n>|\eta|.
\end{equation*}
Hence, for any smooth function $f:{\mathbb{R}}\rightarrow{\mathbb{R}}$ with a Taylor expansion in some domain $D\subset\R$, given by
\begin{equation*}
f(t)=\sum_{n=0}^\infty a_n t^n, \quad t\in D,
\end{equation*}
one can define, for any $u\in{\mathcal{I}}_{0}$ with $u(\Ga_0)\subset D$, the following function on~$\Ga_0$
\begin{equation}  \label{series}
(f^\ast u)(\eta):=\sum_{n=0}^\infty a_n u^{\ast n}(\eta),~\quad
\eta\in\Gamma_0.
\end{equation}
The latter series is finite for all $\eta\in\Gamma_0$.
It is worth noting that, $(f^\ast u)(\emptyset)=a_0$.

In particular, taking $f(t)=e^t$, one can consider, for all  $u\in{\mathcal{I}}_{0}$, the following expression
\begin{equation}  \label{defastexp}
\exp^\ast u(\eta):=\sum_{n=0}^\infty \frac{1}{n!}u^{\ast n}
(\eta)=1^{\ast}(\eta)+\sum_{\bigsqcup\limits_{i}\eta_i=\eta}\prod_i
u(\eta_i),
\end{equation}
where the sum is taking over all partitions of $\eta$ on nonempty sets. Clearly, $k:=\exp^\ast u\in{\mathcal{I}}_1$. The function $u$ is said to be a cumulant of the function $k$.

For any $k\in{\mathcal{I}}_1$, one can consider the function $\bar{k}=k-1^\ast\in{\mathcal{I}}_0$. Then, if we only know that $f:{\mathbb{R}}%
\rightarrow{\mathbb{R}}$ has an expansion
\begin{equation*}
f(1+t)=\sum_{n=0}^\infty a_n t^n, \quad t\in D\subset{\mathbb{R}},
\end{equation*}
we may define
\begin{equation*}
(f^\ast k)(\eta):=\sum_{n=0}^\infty a_n \bar{k}^{\ast n}(\eta),~\quad
\eta\in\Gamma_0.
\end{equation*}
It should be noted again, that, for any $\eta\in\Ga_0$, the latter series is just a finite sum.

The following two examples of such function will be the mostly important for us.
\begin{proposition}\label{inverseast}
Let $k\in{\mathcal{I}}_1$, then there exists the function
\begin{equation}  \label{inverse}
k^{\ast \, -1}(\eta):=\sum_{n=0}^\infty (-1)^n\bar{k}^{\ast n}(\eta), \quad
\eta\in\Gamma_0,
\end{equation}
such that~$k^{\ast\,-1}\in{\mathcal{I}}_1$ and
\begin{equation*}
k\ast k^{\ast \,-1}=1^\ast.
\end{equation*}
\end{proposition}

\begin{proof}
The inclusion $k^{\ast\,-1}\in\I_1$ is followed from \eqref{inverse} directly. Next,
\begin{align*}
\left( k\ast k^{\ast \,-1}\right) \left( \eta \right) &=\sum_{\xi \sqcup
\zeta =\eta}k\left( \xi \right) k^{\ast \,-1}\left( \zeta \right)  =\sum_{\xi \sqcup \zeta =\eta}k\left( \xi \right)
\sum_{n=0}^{\infty
}\left( -1\right) ^{n}\bar{k}^{\ast \,n}\left( \zeta \right)  \\
&=\sum_{\xi \sqcup \zeta =\eta}1^{\ast}\left( \xi \right)
\sum_{n=0}^{\infty}\left( -1\right) ^{n}\bar{k}^{\ast \,n}\left(
\zeta \right) +\sum_{\xi \sqcup \zeta =\eta}\bar{k}\left( \xi
\right) \sum_{n=0}^{\infty}\left( -1\right) ^{n}\bar{k}^{\ast
\,n}\left( \zeta
\right)  \\
&=\sum_{n=0}^{\infty}\left( -1\right) ^{n}\bar{k}^{\ast \,n}\left(
\eta \right) +\sum_{n=0}^{\infty}\left( -1\right) ^{n}\sum_{\xi
\sqcup \zeta
=\eta}\bar{k}\left( \xi \right) \bar{k}^{\ast \,n}\left( \zeta \right)  \\
&=k^{\ast \,-1}\left( \eta \right) +\sum_{n=0}^{\infty}\left(
-1\right)
^{n}\bar{k}^{\ast \,\left( n+1\right)}\left( \eta \right)  \\
&=k^{\ast \,-1}\left( \eta \right) +\sum_{n=1}^{\infty}\left(
-1\right) ^{n-1}\bar{k}^{\ast \,n}\left( \eta \right) =k^{\ast
\,-1}\left( \eta \right) -\sum_{n=1}^{\infty}\left( -1\right)
^{n}\bar{k}^{\ast \,n}\left(
\eta \right)  \\
&=k^{\ast \,-1}\left( \eta \right) -\left( \sum_{m=0}^{\infty}\left(
-1\right) ^{m}\bar{k}^{\ast \,m}\left( \eta \right) -1^{\ast}\left(
\eta \right) \right) =1^{\ast}\left(\eta\right) ,
\end{align*}
which proves the statement.
\end{proof}

For studying the second example, we consider for any $x\in{X}$ the following measurable mapping
\begin{equation}  \label{defder}
({\mathcal{D}}_x G)(\eta):=G(\eta\cup x),\quad G\in L^0(\Ga_0).
\end{equation}
It is easy to check that this mapping is satisfied ``the chain rule'', namely,
\begin{equation}  \label{derpropDx}
{\mathcal{D}}_x(G_1*G_2)=({\mathcal{D}}_x
G_1)*G_2+G_1*({\mathcal{D}}_x G_2), \quad x\in{X},
\end{equation}
for any $G_1, G_2\in L^0(\Ga_0)$. Note that~${\mathcal{D}}_x 1^\ast =0$. Hence, \eqref{defastexp} yields that
\begin{equation}\label{DxOfExp}
{\mathcal{D}}_x\exp^\ast u={\mathcal{D}}_x u\, \ast \exp^\ast u, \quad u\in{%
\mathcal{I}}_0.
\end{equation}

\begin{proposition}
\label{logexp} Let $k\in{\mathcal{I}}_1$. Then, there exists
\begin{equation*}
(\ln^\ast k)(\eta):=\sum_{n=1}^\infty \frac{(-1)^{n-1}}{n} \, \bar{k}^{\ast
n}(\eta),~\quad \eta\in\Gamma_0
\end{equation*}
such that $\ln^\ast k\in{\mathcal{I}}_0$, and, moreover,
\begin{equation*}
\ln^\ast \exp^\ast u=u, \quad u\in{\mathcal{I}}_0,\qquad \qquad
\exp^\ast \ln^\ast k = k, \quad k\in{\mathcal{I}}_1.
\end{equation*}
\end{proposition}

\begin{proof}
The inclusion $\ln^\ast k\in\I_0$ is obvious. Next,
\eqref{derpropDx} and \eqref{inverse} yield that, for all $\eta\in\Ga_0$, $x\in{X}\setminus\eta$,
\[
\D_x \ln^\ast k(\eta)=\D_x k \ast k^{\ast \,-1}.0
\]
Here we used that that ~$\D_x\bar{k}=\D_x k$. Therefore, using \eqref{DxOfExp} we obtain
\begin{align*}
\D_x \ln^\ast\exp^\ast u&=\D_x\exp^\ast u\,\ast
(\exp^\ast u)^{\ast \,-1}\\& = \D_x u\, \ast\exp^\ast
u\,\ast (\exp^\ast u)^{\ast \,-1} =\D_x u.
\end{align*}
On the other hand, if we assume that, for any $u_1, u_2\in\I_0$,
\[
\D_x u_1(\eta)=u_1(\eta\cup x)=\D_x u_2(\eta)=u_2(\eta\cup x), \quad
\eta\in\Ga_0, \ x\in{X}\setminus\eta
\]
then immediately $u_1=u_2$. As a result, $\ln^\ast \exp^\ast u=u$.

Vise versa, let $k\in\I_1$. We set $\exp^\ast \ln^\ast k = k_0$, then $k_0\in \I_1$ and, by the previous considerations, one get
\begin{equation}\label{eqlog}
\ln^\ast k_0= \ln^\ast \exp^\ast \ln^\ast k =\ln^\ast
k.
\end{equation}
Let us prove that this yields $k=k_0$. First of all it should be noted that, for all $k_1, k_2 \in\I_1$, one has $k_1\ast k_2\in\I_1$ and, moreover,
\[
(k_1\ast k_2)^{\ast\,-1}=(k_1)^{\ast\,-1}\ast
(k_1)^{\ast\,-1},
\]
since
$(k_1)^{\ast\,-1}\ast (k_1)^{\ast\,-1} \ast k_1\ast
k_2=1^\ast\ast 1^\ast =1^\ast$.
Next, we have the following
\begin{align*}
\D_x \ln^\ast(k_1\ast k_2)&=(k_1\ast k_2)^{\ast\,-1}\ast\D_x(k_1\ast
k_2)\\&=k_1^{\ast\,-1}\ast k_2^{\ast\,-1}\ast \D_x k_1\ast
k_2+k_1^{\ast\,-1}\ast k_2^{\ast\,-1}\ast k_1\ast  \D_x
k_2\\&=k_1^{\ast\,-1}\ast \D_x k_1+k_2^{\ast\,-1}\ast  \D_x k_2=\D_x \ln^\ast k_1 +\D_x \ln^\ast k_2.
\end{align*}
Therefore,
$\ln^\ast(k_1\ast k_2)=\ln^\ast k_1 +\ln^\ast k_2$.
Hence,
\[
0=\ln^\ast 1^\ast=\ln^\ast (k_2\ast k_2^{\ast\, -1})=\ln^\ast k_2+\ln^\ast k_2^{\ast\, -1},\qquad
\ln^\ast k_2^{\ast\, -1} = -\ln^\ast k_2,
\]
that yields $\ln^\ast(k_1\ast k_2^{\ast\, -1} )=\ln^\ast k_1 -\ln^\ast k_2$. As a result, \eqref{eqlog} implies
\begin{equation}\label{eq0}
\ln^\ast(k\ast k_0^{\ast\, -1} )=0.
\end{equation}
On the other hand, for any $k_3\in\I_1$, the condition $\ln^\ast k_3=0$
yields
\[
0=\D_x\ln^\ast k_3=k_3^{\ast\,-1}\ast \D_x k_3,
\]
that gives $0=\D_x k_3 (\eta)=k_3(\eta\cup x)$, $k_3=1^\ast$. Then, by~\eqref{eq0}, we get
$k\ast k_0^{\ast\, -1}=1^\ast$, $k_0=k$
which proves the assertion.
\end{proof}

\section{A multiplication operator with respect to $*$-convolution}

Let $a\in\K_{C_a,\delta_a}$ for arbitrary $C_a>0$, $\delta_a\geq0$. Then, by~Proposition~\ref{convest}, for any $C>C_a$, $\delta\geq\delta_a$, one can consider the mapping $A:\K_{C-C_a,\delta}\to\K_{C,\delta}$ given by the equality
\begin{equation}\label{convoper}
    (Ak)(\eta)=(a*k)(\eta), \quad \eta\in\Ga_0.
\end{equation}
\begin{proposition}
The operator $A$ with domain $\K_{C-C_a,\delta}$ is closable in the Banach space $\K_{C,\delta}$.
\end{proposition}
\begin{remark}
  It is easily seen that the operator $A$ is not densely defined in~$\K_{C,\delta}$.
\end{remark}
\begin{proof}
Let $\{k_n\}_{n\in\N}\subset \K_{C-C_a,\delta}$ and $\|k_n\|_{{C,\delta}}\to 0$, $n\to\infty$. Suppose that there exists $b\in\K_{C,\delta}$ such that
$\|a*k_n-b\|_{{C,\delta}}\to 0$. Then, by~Proposition~\ref{convest} and inequalities between the norms in $\K_{C,\delta}$ and $\K_{C+C_a,\delta}\supset\K_{C,\delta}\ni b$, we obtain
\begin{align*}
  \|b\|_{{C+C_a,\delta}}&\leq \|a*k_n\|_{{C+C_a,\delta}}+ \|a*k_n-b\|_{{C+C_a,\delta}}\\& \leq \|a\|_{{C_a,\delta_a}}\cdot\|k_n\|_{{C,\delta}}+\|a*k_n-b\|_{{C,\delta}}\to0,\quad
n\to\infty.
\end{align*}
Therefore, $b=0$ in~$\K_{C+C_a,\delta}$ that yields $b(\eta)=0$ for $\la$-a.a. $\eta\in\Ga_0$, hence, $b=0$ in $\K_{C,\delta}$ too.
\end{proof}

It is worth noting that if $a\in L^\infty(\Ga_0)$, then, by~Proposition~\eqref{convest}, the operator \eqref{convoper} is well-defined on the whole space $\K_{C,\delta}$, for any $C>1$, $\delta\geq 1$, therefore, it is bounded in this space.

We consider the evolution equation
\begin{equation}\label{convevol}
    \frac{\partial}{\partial t} k_t=Ak_t, \qquad k\bigr|_{t=0}=k_0.
\end{equation}
It is straightforward that the following function is an informal solution to \eqref{convevol}
\begin{equation}\label{formsol}
    k_t=\sum_{n=0}^\infty\frac{t^n}{n!}a^{*n}*k_0=\exp^*(ta)*k_0.
\end{equation}

If $a\in\I_0$ then~$\exp^*(ta)$ is point-wise defined (see \eqref{defastexp}) and, therefore, \eqref{formsol} gives a point-wise solution to \eqref{convevol}.

If $a\in L^\infty(\Ga_0)$, then, by~Corollary~\ref{cor1}, $a^{*n}\in\K_{C,0}$ for any $C\geq2$, moreover,
\[
\|\exp^*(ta)\|_{C,0}\leq 1 + t\|a\|_{L^\infty(\Ga_0)}+\sum_{n=2}\frac{t^n}{n!}\biggl(\frac{C}{C-1}\biggr)^{n-2}\|a\|_{L^\infty(\Ga_0)}^n
<\exp\biggl(\frac{Ct}{C-1}\|a\|_{L^\infty(\Ga_0)}\biggr),
\]
that yields $\exp^*(ta)\in\K_{C,0}$, $C\geq2$. Then, directly by~Proposition~\ref{convest}, the equation \eqref{convevol} has a solution in the spaces $\K_{C,\delta}$, $\delta\geq0$.

If one would like to consider solutions to \eqref{convevol} in wider spaces, for $\delta\geq1$, then one can allow $a\in\K_{C_a,\delta_a}$, $\delta_a\geq1$. In this case, by~Corollary~\ref{cor1}, $a^{*n}\in\K_{C,\delta_a}$ for any $C>C_a$, hence, the series in~\eqref{formsol} converges in~$\K_{C,\delta_a}$. Then, again by~Corollary~\ref{convest}, one get that e.g. $k_0\in\K_{C_0,\delta_a}$, $C_0<C$ yields $k_t\in\K_{C,\delta_a}$.

Let us consider the following Banach space $\mathcal{L}_{C,\delta}:=L^1\bigl(\Ga_0, C^{|\eta|}(|\eta|!)^\delta\,d\la(\eta)\bigr)$, $C>0$, $\delta\geq0$ with norm
\[
\|G\|_{\mathcal{L}_{C,\delta}}:=\int_{\Ga_0}|G(\eta)|C^{|\eta|}(|\eta|!)^\delta\,d\la(\eta)
=\sum_{n=0}^\infty\frac{C^n}{(n!)^{1-\delta}}\int_{{X}^n}|G^{(n)}(x_1,\ldots,x_n)|dm(x_1)\ldots dm(x_n).
\]
Clearly, $\Bbs\subset\mathcal{L}_{C,\delta}$ for all $C>0$, $\delta\geq0$, and the inclusions are dense. Note also that $e_\la(f)\in\mathcal{L}_{C,\delta}$ for all $C>0$, $\delta\in[0;1)$, $f\in L^1({X},dm)$.

The space $\K_{C,\delta}$ is a realization of a space which is topologically dual to~$\mathcal{L}_{C,\delta}$. Therefore, one can consider a duality between this two spaces which is given by the pairing
\begin{equation*}
    \langle\!\langle G, k\rangle\!\rangle:=\int_{\Ga_0} G(\eta) k(\eta)d\la(\eta), \quad
G\in\mathcal{L}_{C,\delta}, k\in\K_{C,\delta}.
\end{equation*}

Let the operator $A'$ in $\mathcal{L}_{C,\delta}$ is given by
\begin{equation*}
    (A'G)(\eta):=\int_{\Ga_0}G(\eta\cup\xi)a(\xi)\,d\la(\xi), \quad G\in D(A'),
\end{equation*}
and $D(A')$ consists of all $G\in\mathcal{L}_{C,\delta}$, such that~$A'G\in\mathcal{L}_{C,\delta}$. Evidently, this operator with a maximal domain is closed. By~\eqref{minlosid-ast}, we obtain
\[
\int_{\Ga_0}|A'G(\eta)|C^{|\eta|}(|\eta|!)^\delta\,d\la(\eta)
\leq \int_{\Ga_0}\bigl|G(\eta)\bigr| (|a|*C^{|\cdot|}(|\cdot|!)^\delta)(\eta)\,d\la(\eta).
\]
Then, Proposition~\ref{convest} yields that for all $a\in\K_{C_a,\delta_a}$, $C_a>0$, $\delta_a\geq0$ one has the following inclusion $\Bbs\subset D(A')$ as $C>0$, $\delta\geq0$. Therefore, $A'$ is densely defined. Moreover, for any $\delta_a\leq\delta$, one has the inclusion $\mathcal{L}_{C+C_a,\delta}\subset D(A')$. On the other hand, for $\max\{1,\delta_a\}\leq\delta$, $C_a<C$, one get~$D(A')=\mathcal{L}_{C,\delta}$, hence, the operator $A'$ is bounded in~$\mathcal{L}_{C,\delta}$.

By \eqref{minlosid-ast}, for any $G\in D(A')\subset \mathcal{L}_{C,\delta}$, $C>0$, $\delta\geq0$ and any $k\in\K_{C,\delta}$ with $Ak\in\K_{C,\delta}$, one has
\[
\langle\!\langle A'G, k\rangle\!\rangle=\langle\!\langle G, Ak\rangle\!\rangle.
\]
The operator $A'$ is said to be pre-dual to~$A$.

\begin{proposition}
Let $a\in\K_{C_a,\delta_a}$, $C_a>0$, $\delta_a\geq0$, $C>C_a$, $\delta\geq\max\{\delta_a,1\}$. Then there exists $z_0>0$ such that for all $z>z_0$ the resolvent of the operator $A'$ in the space $\mathcal{L}_{C,\delta}$ has the form
\begin{equation}\label{rez}
    \bigl(R_z(A')G\bigr)(\eta):=\bigl((z\1-A')^{-1}G\bigr)(\eta)=\sum_{n=0}^\infty\frac{1}{z^{n+1}}
\int_{\Ga_0}G(\eta\cup\xi)a^{*n}(\xi)\,d\la(\xi).
\end{equation}
\end{proposition}
\begin{proof}
We first show that \eqref{rez} is a Neumann series. Indeed, $(z\1-A')^{-1}=z^{-1}\sum_{n=0}^\infty \frac{(A')^n}{z^n}$ and, using \eqref{minlosid-ast}, $(A')^nG(\eta)=\int_{\Ga_0}G(\eta\cup\xi)a^{*n}(\xi)\,d\la(\xi)$. Since $A'$ is a bounded operator in~$\mathcal{L}_{C,\delta}$, the assertion is proved.
\end{proof}
\begin{remark}
Let $a\in\I_0$. Then, for any $z\in\R$ and $k\in L^0(\Ga_0)$, there exists
\[
(z\1-A)^{-1}k=\frac{1}{z}\Bigl(1^*-\frac{a}{z}\Bigr)^{*-1}*k=\sum_{n=0}^\infty\frac{1}{z^{n+1}}a^{*n}*k,
\]
and the series is point-wise defined.
\end{remark}

We will consider three simple but important examples of a multiplication operator $A$, note that $a\in L^\infty(\Ga_0)$ in all the cases. Let $a(\eta)=1$, $\eta\in\Ga_0$. Then $Ak=K_0k$, where
\[
(K_0k)(\eta)=\sum_{\xi\subset\eta}k(\xi), \quad \eta\in\Ga_0
\]
(the meaning of the notation $K_0$ will be clear from the second part of the paper). The pre-dual operator to~$K_0$ is the so-called Mayer operator
\[
(DG)(\eta):=(K_0'G)(\eta)=\int_{\Ga_0}G(\eta\cup\xi)\,d\la(\xi), \quad \eta\in\Ga_0.
\]
Since $1=e_\la(1)$, the equality \eqref{addexp} yields $a^{*n}(\eta)=n^{|\eta|}$, $\eta\in\Ga_0$. Therefore, an informal solution to the evolution equation \eqref{convevol} is
\[
k_t=\sum_{n=0}^{\infty}\frac{n^{|\cdot|}t^n}{n!}*k_0,
\]
and, evidently, the series converges point-wise.

The second example is the case $a(\eta)=-1$, $\eta\in\Ga_0$. This defines, of course, the inverse operator
\[
(K_0^{-1}k)(\eta)=\sum_{\xi\subset\eta}(-1)^{|\eta\setminus\xi|}k(\xi), \quad \eta\in\Ga_0,
\]
since $\bigl(1*(-1)\bigr)(\eta)=\sum_{\xi\subset\eta}(-1)^{|\eta\setminus\xi|}=0^{|\eta|}=1^*(\eta)$.
In this case, the pre-dual operator is
\[
(D^{-1}G)(\eta):=((K_0^{-1})'G)(\eta)=\int_{\Ga_0}(-1)^{|\xi|}G(\eta\cup\xi)\,d\la(\xi), \quad \eta\in\Ga_0.
\]
The solution to the equation \eqref{convevol} is given by analogy, because of~$(-1)^{*n}(\eta)=(-1)^n n^{|\eta|}$, $\eta\in\Ga_0$.

Finally, let, $\sigma:{X}\to\R$ be a measurable function and
\[
a(\eta)=
\begin{cases}
\sigma(x), &\eta=\{x\},\\0,&|\eta|\neq1,
\end{cases}
\]
$\eta\in\Ga_0$. Then
\[
(Ak)(\eta)=\sum_{x\in\eta}\sigma(x) k(\eta\setminus x),\quad\eta\in\Ga_0.
\]
From the definition of $*$-convolution we obtain by the induction principle that $a^{*n}(\eta)=n!\1_{\Ga^{(n)}}(\eta)\prod_{x\in\eta}\sigma(x)$, $\eta\in\Ga_0$, $n\in\N$. Therefore,
\[
\exp^* (ta) (\eta)=\sum_{n=0}^\infty\1_{\Ga^{(n)}}(\eta)t^n\prod_{x\in\eta}\sigma(x)=e_\la(t\sigma,\eta), \quad \eta\in\Ga_0.
\]
Hence, the point-wise solution to the evolution equation
\begin{equation}\label{ev-ind}
\frac{\partial}{\partial t}k_t(\eta)=\sum_{x\in\eta}\sigma(x) k_t(\eta\setminus x),\quad k\bigr|_{t=0}=k_0,\quad\eta\in\Ga_0
\end{equation}
is the function $k_t=e_\la(t\sigma)*k_0$. It is worth noting, that, by \eqref{addexp}, $k_0=e_\la(C)\in\K_{C,0}$, $C>0$ yields $k_t(\eta)=e_\la(C+t\sigma,\eta)$, $\eta\in\Ga_0$. Therefore, if e.g. $\sigma\in L^\infty({X},dm)$, $q=\|\sigma\|_{L^\infty({X})}$, then~$k_t\in\K_{C+tq,0}$, $t\geq0$. This means that for any $C'>C$ the solution to \eqref{ev-ind} belongs to the space $\K_{C',0}$ on a finite time interval only. On the other hand, it is easily seen that for all $C'>0$, $\delta>0$, $t\geq0$ the inclusion $k_t\in\K_{C',\delta}$ holds true. It can be shown by analogy that $k_0\in\K_{C,\delta}$, $C>0$, $\delta>0$ implies $k_t\in\K_{C,\delta+\eps}$ for all $\eps>0$ and $t\geq0$.

\begin{remark}
Let $\delta\in[0;1)$. In the latter example the evolution $\K_{C,\delta}\ni k_0\mapsto k_t\in\K_{C,\delta+\eps}$ with an arbitrary $\eps>0$ can be constructed only by using the explicit expression for $\exp^*(ta)$. If we would like to obtain an estimate for~$\exp^*(ta)$ with an arbitrary $a$ using the series, then we will need to consider $\eps\geq1$. The problem is that to include $a^{*n}\in\K_{Cn,\delta}$ into the space $\K_{C,\delta+\eps}$ with $\eps$ independent on $n$, the norm of $a^{*n}$ in~$\K_{C,\delta+\eps}$ will be increase in~$n$ depending on $\eps$. Unfortunately, at present, it is known only the upper bound by the expression $\|a\|_{C,\delta}^n\exp\bigl\{\eps n^{\frac{1}{\eps}}\bigr\}$, that implies that the condition $\eps\geq1$ is sufficient for the convergence of the series $\sum_{n=0}^\infty \frac{t^n}{n!}a^{*n}$ in~$\K_{C,\delta+\eps}$. The exact asymptotic of an inclusion operator in~$n$ and~$\eps$ is unknown.
\end{remark}

\begin{remark}
Let $\D(\Ga_0)$ be a linear topological space of measurable functions on~$\Ga_0$, that is continuously embedded into~$\mathcal{L}_{C,\delta}$ for some $C>0$, $\delta\geq0$. For any $k\in\K_{C,\delta}$, the mapping $G\mapsto\int_{\Ga_0} G k\,d\la$ defines a linear continuous functional on~$\mathcal{L}_{C,\delta}$; therefore, this mapping defines a liner continuous functional on~$\D(\Ga_0)$ too. Therefore, $k$ can be considered as a regular generalized function on~$\D(\Ga_0)$. In this case, the equality \eqref{minlosid-ast} can be considered as a way to define a convolution for regular generalized functions, cf. e.g.~\cite[p.~103]{GS1964}. By associativity of $*$-convolution the operator $A$ has the following property: $A(k_1*k_2)=(Ak_1)*k_2=k_1*(Ak_2)$. An arbitrary operator on generalized functions over $(\R^d)^n$ has the same property, see e.g. \cite[p.~105]{GS1964}. However, $A$ is not satisfied to the chain rule in an algebra of function from~$L^0(\Ga_0)$ with a product given by the $*$-convolution. Derivation operators with respect to the $*$-convolution are considered in the sequel.
\end{remark}

\section{Some additional constructions}

\subsection{Convolutions of measures on $\Ga_0$}

In what follows we will need spaces of configurations of two different point types, which we denote ``$+$'' and ``$-$''. Namely, for any $Y^\pm\in\B({X})$, $n^\pm\in\N$
we consider $\Ga_{0,Y^\pm}^{\pm,(n^\pm)}:=\Ga_{0,Y^\pm}^{(n^\pm)}$, $\Ga_{0,Y^\pm}^\pm:=\Ga_{0,Y^\pm}$, $\Ga_0^\pm:=\Ga_0$ and we set $\Ga_{0,Y^+,Y^-}^{2,(n^+,n^-)}:=\Ga_{0,Y^+}^{+,(n^+)}\times\Ga_{0,Y^-}^{-,(n^-)}$, $\Ga^2_{0,Y^+,Y^-}:=\Ga_{0,Y^+}^+\times\Ga_{0,Y^-}^-$, $\Ga_0^2:=\Ga_0^+\times\Ga_0^-$.
In the case $Y^+=Y^-=Y\in\B({X})$, $n^+=n^-=n\in\N$ we will write just $\Ga_{0,Y}^{2,(n)}=\Ga_{0,Y}^{+,(n)}\times\Ga_{0,Y}^{-,(n)}$, $\Ga_{0,Y}^2=\Ga_{0,Y}^+\times\Ga_{0,Y}^-$. On the all spaces above product-topologies can be considered. These topologies will be well-correspond with expansions like $\Ga_{0,Y}^2=\bigsqcup_{n^+,n^-\in\N_0}\Ga_{0,Y}^{+,(n^+)}\times\Ga_{0,Y}^{-,(n^-)}$.
Clearly, the corresponding Borel $\sigma$-algebras will be minimal $\sigma$-algebras which  are generated by Cartesian products of Borel subsets of the configuration spaces of each type. As before, we will omit the subscript $0$ if only $Y=\La\in\Bb$.

Let us define also some notions by analogy with one-type configuration spaces. A function ${G}:\Ga_0^2\to\R$ is said to have a local support if there exists $\La\in\Bb$such that ${G}\upharpoonright_{\Ga_0^2\setminus(\Ga_\La^+\times
\Ga_\La^-)}=0$. Let $L^0_\ls(\Ga_0^2)$ denote the class of all measurable functions on~$\Ga_0^2$ which have local supports. A set ${B}\in \B(\Ga^2_0)$ is said to be bounded if there exist $\La\in\Bb$ and $N\in\N$ such that
${B}\subset\Bigl(\bigsqcup_{n=0}^N\Ga_\La^{+,(n)}\Bigr)
\times\Bigl(\bigsqcup_{n=0}^N\Ga_\La^{-,(n)}\Bigr)$.
Let $\BBg$ denote the class of all bounded subsets from~$\B(\Ga^2_0)$.
A function ${G}:\Ga_0^2\to\R$ is said to have a bounded support if there exists ${B}\in\BBg$ such that
${G}\upharpoonright_{\Ga_0^2\setminus\widetilde{B}}=0$.
Let $\BBs$ denote the class of all bounded functions on $\Ga^2_0$ which have bounded supports. A measure ${\rho}$ on $\bigl(\Ga^2_0,\B(\Ga^2_0)\bigr)$
is said to be a locally finite measure if ${\rho}({B})<\infty$, for all ${B}\in\BBg$. Let $\MLf$ denote the class of all such measures.

For an arbitrary measurable function $G:\Ga_0\to\R$, we consider the measurable function $\widetilde{G}:\Ga_0^2\to\R$ given by
\begin{equation}  \label{tildeG}
    \widetilde{G}(\eta^+,\eta^-)=
    G(\eta^+\cup\eta^-), \qquad (\eta^+,\eta^-)\in\Ga_0^2.
\end{equation}

For $\rho_i\in\Mlf$, $i=1,2$, we define a $\widehat{\rho}$ on~$\bigl(\Ga^2_0,\B(\Ga^2_0)\bigr)$ given by
$d\widehat{\rho}(\eta^+,\eta^-)=d\rho_1(\eta^+)\,d\rho_2(\eta^-)$.
On the other words, $\widehat{\rho}=\rho_1\otimes\rho_2$. Clearly,~$\widehat{\rho}\in\MLf$.

\begin{definition}
Let $\rho_i$, $i=1,2$ be measures on $\bigl(\Ga_0,\B(\Ga_0)\bigr)$. A measure $\rho$ on
$\bigl(\Ga_0,\B(\Ga_0)\bigr)$ is said to be the convolution of these measures if, for any $G:\Ga_0\to\R$ such that $\widetilde{G}\in L^1(\Ga_0^2,d\widehat{\rho})$, the following identity holds true
\begin{equation}
\int_{\Ga_0} G(\eta)d\rho(\eta)
=\int_{\Ga_0^2} \widetilde{G}(\eta^{+},\eta^-)
\,d\widehat{\rho}(\eta^+,\eta^-)
=\int_{\Ga_0^+} \int_{\Ga_0^-}G(\eta^{+}\cup\eta^-)
\,d\rho_1(\eta^+)\,d\rho_2(\eta^-).\label{convmeasfin}
\end{equation}
The notation is $\rho=\rho_1\ast\rho_2$.
\end{definition}

\begin{proposition}
Let $\rho_{1,2}\in\Mlf$, $\rho=\rho_1\ast\rho_2$.
Then $\rho\in\Mlf$.
\end{proposition}
\begin{proof}
Let $B\in\Bbg$, hence, there exist $\La\in\Bb$ and~$N\in\N$ such that
$B\subset A_N:=\bigcup\limits_{n=0}^{N}\Ga^{(n)}_\La$. Then
\begin{align*}
\rho(B)&=\int_\Ga \1_B(\eta)\,d\rho(\eta)=\int_{\Ga_0^+}
\int_{\Ga_0^-}\1_B(\eta^{+}\cup\eta^-)
\,d\rho_1(\eta^+)\,d\rho_2(\eta^-)\\
&\leq\int_{\Ga_0^+} \int_{\Ga_0^-}\1_{A_N}(\eta^{+}\cup\eta^-)
\,d\rho_1(\eta^+)\,d\rho_2(\eta^-)\\
&\leq\int_{\Ga_0^+} \int_{\Ga_0^-}\1_{A_N}(\eta^{+})\1_{A_N}(\eta^-)
\,d\rho_1(\eta^+)\,d\rho_2(\eta^-)=\rho_1(A_N)\rho_2(A_N)<\infty.
\end{align*}
The statement is proved.
\end{proof}

The notation ``$\ast$'' for the convolution of measures coincides with the notation for the $*$-convolution of functions given by~\eqref{ast}. This is motivated by the following statement.
\begin{proposition} \label{convofcorfunc}
Let $\rho_i\in\Mlf$, $i=1,2$. Suppose that there exist the following Radon--Nikodym derivatives with respect to the Lebesgue--Poisson measure:
$k_i=\dfrac{d\rho_i}{d\lambda}$, $i=1,2$.
Then, the convolution of measures, $\rho=\rho_1\ast\rho_2$, has also a Radon--Nikodym derivative with respect to the Lebesgue--Poisson measure, $k=\dfrac{d\rho}{d\lambda}$, and, moreover, $k=k_1 * k_2$.
\end{proposition}

\begin{proof}
Let $G\in \Bbs$. By~\eqref{convmeasfin}, one has
\begin{align*}
\int_{\Ga_0}G(\eta)\,d\rho(\eta)&=\int_{\Ga_0^+}
\int_{\Ga_0^-}G(\eta^{+}\cup\eta^-)
\,d\rho_1(\eta^+)\,d\rho_2(\eta^-)\\
&=\int_{\Ga_0^+} \int_{\Ga_0^-}G(\eta^{+}\cup\eta^-)
k_1(\eta^+)k_2(\eta^-)\,d\la(\eta^+)
\,d\la(\eta^-)\\
&=\int_{\Ga_0}G(\eta)(k_1\ast k_2)(\eta)\,d\la(\eta),
\end{align*}
where we used \eqref{minlosid-ast}. The statement is proved.
\end{proof}

\subsection{Generating functionals}

Generating functionals, a.k.a. Bogolyubov functionals, were introduced in 1946, see \cite{Bog1962}, the more recent results see e.g. in~\cite{KKO2006}. Properties of generating functionals are closely connected to properties of probability measures on spaces of locally finite configurations. In spite of this, in the first part of our work we restrict our attention to the properties of the generating functionals in the framework of spaces of finite configuration only.

Let $k\in\K_{C,\delta}$, $C>0$, $\delta\in[0;1)$. Then the functional (cf.~\cite[expr. ~(9)]{KKO2006})
\begin{equation*}
    B_k(f):=\int_{\Ga_0}e_\la(f,\eta) k(\eta)\,d\la(\eta), \quad f\in L^1:=L^1({X},dm)
\end{equation*}
is well-defined since
\[
\bigl\vert B_k(f)\bigr\vert\leq \sum_{n=0}^\infty \frac{1}{(n!)^{1-\delta}}\bigl(C\|f\|_{L^1}\bigr)^n<\infty.
\]
By \eqref{minlosid-ast} and Proposition~\ref{convest}, one get that $k_i\in\K_{C_i,\delta_i}$, $C_i>0$, $\delta_i\in[0;1)$, $i=1,2$ yield $k_1*k_2\in\K_{C,\delta}$, where $C=C_1+C_2$, $\delta=\max\{\delta_1,\delta_2\}$ and
\[
B_{k_1*k_2}(f)=B_{k_1}(f)B_{k_2}(f), \quad f\in L^1.
\]

\begin{remark}
This procedure might be easily generalized to the case of measures on~$\Ga_0$. Namely, let $\rho\in\Mlf$ be such that $e_\la(f)\in L^1(\Ga_0,d\rho)$ for all $f\in L^1$. Then, one can define the functional
\[
\tilde{B}_\rho(f):=\int_{\Ga_0}e_\la(f,\eta)\,d\rho(\eta), \quad f\in L^1.
\]
By~\eqref{convmeasfin}, we get $\tilde{B}_{\rho_1\ast\rho_2}(f)=\tilde{B}_{\rho_1}(f)\tilde{B}_{\rho_2}(f)$, $f\in L^1$. Clearly, if only $k=\dfrac{d\rho}{d\la}\geq0$ exists then $B_k=\tilde{B}_\rho$.
\end{remark}

\begin{proposition}
Let $u\in\I_0$ and suppose that there exist $C,C'>0$, $\delta,\delta'\in[0;1)$ such that $u\in\K_{C,\delta}$, $\exp^*|u|\in\K_{C',\delta'}$. Then $B_k(f)>0$ for $k=\exp^*u$ and for all $f\in L^1$.
\end{proposition}
\begin{proof}
Suppose that $u\in\K_{C,\delta}$, then $|B_u(f)|\leq B_{|u|}(|f|)<\infty$. Therefore, by \eqref{minlosid-ast}, we obtain
\[
\int_{\Ga_0}e_\la(|f|)|u|^{*n}\,d\la=\bigl(B_{|u|}(|f|)\bigr)^n<\infty.
\]
Hence,
\begin{equation}\label{adw3}
    \Biggl\vert\sum_{n=0}^\infty \frac{1}{n!}\int_{\Ga_0}e_\la(f)u^{*n}\,d\la\Biggr\vert
\leq \exp\bigl(B_{|u|}(|f|)\bigr)<\infty.
\end{equation}
Set $g_N:=\sum_{n=0}^N \frac{1}{n!}\int_{\Ga_0}e_\la(f)u^{*n}\,d\la\in\R$. By~\eqref{adw3}, we get that there exists a finite limit $\lim_{N\to\infty}g_N$. Next, the sequence $U_N:=\sum_{n=0}^N \frac{1}{n!} e_\la(f) u^{*n}$ has an integrable dominated function $e_\la(|f|)\exp^*|u|$ in the space $L^1(\Ga_0,d\la)$, since $\exp^*|u|\in\K_{C',\delta'}$. As a result, by the dominated convergence theorem, one has
\begin{align*}
B_k(f)&=\int_{\Gamma _{0}}e_{\lambda}\left( f,\eta \right) k\left( \eta
\right)
d\lambda \left( \eta \right)  =\int_{\Gamma _{0}}e_{\lambda}\left( f,\eta \right) \sum_{n=0}^{\infty}
\frac{1}{n!}u^{\ast n}d\lambda \left( \eta \right) \\
&=\sum_{n=0}^{\infty}\frac{1}{n!}\int_{\Gamma _{0}}e_{\lambda
}\left(
f,\eta \right) u^{\ast n}\left( \eta \right) d\lambda \left( \eta \right) =\exp \left\{ \int_{\Gamma _{0}}e_{\lambda}\left( f,\eta \right)
u\left( \eta \right) d\lambda \left( \eta \right) \right\} > 0,
\end{align*}
that proves the assertion.
\end{proof}

\begin{remark}
It is worth noting that for any $k\in\I_1$ there always exists $u\in\I_0$ such that $k=\exp^*u$. Therefore, $B_k$ is always a positive functional if only the sufficient conditions on the growth of $|u|$ and~$\exp^*|u|$ hold.
\end{remark}

\subsection{Derivation operator with respect to the $*$-convolution}
As was noted before, the operator $D_x$, given by~\eqref{defder}, is satisfied the chain rule with respect to the $*$-convolution, see \eqref{derpropDx}. Let us consider an another operator with such a property. Let $(Nk)(\eta)=|\eta|k(\eta)$, $k\in L^0(\Ga_0)$, $\eta\in\Ga_0$. Then
\begin{align*}
\bigl(N(k_1*k_2)\bigr)(\eta)&=|\eta|\sum_{\xi\subset\eta}k_1(\xi)k_2(\eta\setminus\xi)=\sum_{\xi\subset\eta}|\xi|k_1(\xi)k_2(\eta\setminus\xi)
+\sum_{\xi\subset\eta}k_1(\xi)|\eta\setminus\xi|k_2(\eta\setminus\xi)\\&=
\bigl((Nk_1)*k_2\bigr)(\eta)+\bigl(k_1*(Nk_2)\bigr)(\eta)
\end{align*}
for all $k_1,k_2\in L^0(\Ga_0)$, $\eta\in\Ga_0$.

\begin{definition}
An operator $B$ on~$L^0(\Ga_0)$ is said to be a derivation operator if $B1^*=0$ and
\begin{equation}\label{od}
    \bigl(B(k_1*k_2)\bigr)(\eta)=\bigl((Bk_1)*k_2\bigr)(\eta)+\bigl(k_1*(Bk_2)\bigr)(\eta)
\end{equation}
for $\la$-a.a. $\eta\in\Ga_0$.
\end{definition}
Note that as yet we consider these operators point-wise defined only, without any relation to some Banach spaces.

Therefore, operators $\D_x$ and~$N$ are derivation operators (since equalities $D_x1^*=N1^*=0$ are followed by definitions of these operators). A number of other examples of such operators we consider in the second part of this work.

By an induction principle, $Bu^{*n}=n(Bu)*u^{*(n-1)}$, $n\in\N$, $u\in L^0(\Ga_0)$. Then, for any $u\in\I_0$ the following (point-wise) equality holds, cf.~\eqref{DxOfExp},
\begin{equation}\label{difofexp}
    B\exp^*u=B\biggl(1^*+\sum_{n=1}^\infty \frac{1}{n!}u^{*n}\biggr)=\sum_{n=1}^\infty \frac{1}{n!}n(Bu)*u^{*(n-1)}=(Bu)*\exp^*u.
\end{equation}
The equality \eqref{difofexp} has an important corollary. Let $B$ be an derivation operator and consider the evolution equation
\[
\frac{\partial}{\partial t} k_t=Bk_t, \quad k\bigr|_{t=0}=k_0.
\]
Suppose that $k_t(\emptyset)=1$, $t\geq0$, this yields $k_t\in\I_1$. Then, by~Proposition~\ref{logexp}, for any $t\geq0$, there exists $u_t\in\I_0$ such that $k_t=\exp^*u_t$. By \eqref{difofexp}, we obtain
\begin{equation}\label{adw1}
    \frac{\partial}{\partial t} k_t=B\exp^*u_t=(Bu_t)*k_t.
\end{equation}
On the other hand, \eqref{defastexp} directly implies that, by analogy to~\eqref{difofexp},
\begin{equation}\label{adw2}
    \frac{\partial}{\partial t} k_t=\frac{\partial}{\partial t} \exp^*u_t=\frac{\partial}{\partial t} u_t*\exp^*u_t=\frac{\partial}{\partial t} u_t*k_t.
\end{equation}
By our assumption $k_t\in\I_1$, then Proposition~\ref{inverseast} yields that there exist $k_t^{*-1}\in\I_1$. If we compare now the right hand sides of \eqref{adw1} and \eqref{adw2}, and multiply them (in the sense of the $*$-convolution) on~$k_t^{*-1}$, we obtain
\[
\frac{\partial}{\partial t} u_t=Bu_t.
\]
As a result, the equation for cumulants $u_t$ coincides with the equation for functions $k_t$.

\begin{proposition}
Let $(B, D(B))$ be an operator in $\K_{C,\delta}$, $C>0$, $\delta\geq0$ with the maximal domain. Let $\bigl(B', D(B')\bigr)$ be a closed densely defined operator in $\mathcal{L}_{C,\delta}$ such that $\langle\!\langle B'G, k\rangle\!\rangle=\langle\!\langle G, Bk\rangle\!\rangle$ for all $G\in D(B')$, $k\in D(B)$. Suppose also that $G(\cdot\cup\eta)\in D(B')$ for $\la$-a.a. $\eta\in\Ga_0$ and for all $G\in D(B')$, and that, for $\la$-a.a. $\eta,\xi\in\Ga_0$,
\begin{equation}  \label{dual-sum}
(B'G)(\eta\cup\xi) =\bigl((B 'G)(\cdot\cup\xi)\bigr)(\eta)+\bigl((B ' G)(\cdot\cup\eta)\bigr)(\xi).
\end{equation}
Then, for all $k_1,k_2\in D(B)$ with $k_1*k_2\in D(B)$, $k_1*(Ak_2),(Ak_1)*k_2\in\K_{C,\delta}$, the equality \eqref{od} holds.
\end{proposition}
\begin{proof}
By \eqref{minlosid-ast} and \eqref{dual-sum}, for all $G,k_1,k_2$ as above, one has
\begin{align*}
&\int_{\Ga_0} G(\eta) \bigl(B(k_1*k_2)\bigr)(\eta)d\la(\eta)=\int_{\Ga_0} (B'G)(\eta) (k_1*k_2)(\eta)d\la(\eta)\\=&\int_{\Ga_0}\int_{\Ga_0} (B'G)(\eta\cup\xi)  k_1(\eta)k_2(\xi)d\la(\eta)d\la(\xi)\\
=&\int_{\Ga_0}\int_{\Ga_0} \bigl(B'G(\cdot\cup\xi)\bigr) (\eta)k_1(\eta) k_2(\xi)d\la(\eta)d\la(\xi)\\&\quad+\int_{\Ga_0}\int_{\Ga_0} \bigl(B'G(\cdot\cup\eta)\bigr) (\xi) k_1(\eta) k_2(\xi)d\la(\eta)d\la(\xi)\\
=&\int_{\Ga_0}\int_{\Ga_0} G(\eta\cup\xi) (Bk_1)(\eta) k_2(\xi)d\la(\eta)d\la(\xi)\\&\quad+\int_{\Ga_0}\int_{\Ga_0} G(\eta\cup\xi) k_1(\eta)(B k_2)(\xi)d\la(\eta)d\la(\xi)\\=&\int_{\Ga_0} G(\eta) \bigl((Bk_1)\ast k_2\bigr)(\eta)d\la(\eta)+\int_{\Ga_0} G(\eta) \bigl(k_1\ast (B k_2)\bigr)(\eta)d\la(\eta),
\end{align*}
which proves the assertion.
\end{proof}

\subsection{$\star$-convolution of functions on~$\Ga_0$}
The following convolution between functions on~$\Ga_0$ was introduced in~\cite{KK2002}.
\begin{definition}
Set, for arbitrary measurable functions $G_1$ and~$G_2$ on
$\Ga_0$,
\begin{equation}\label{star}
 (G_1\star G_2)(\eta):=
 \sum_{\xi_1\sqcup\xi_2\sqcup\xi_3=\eta}
 G_1(\xi_1\cup\xi_2)\,G_2(\xi_2\cup\xi_3), \quad \eta\in\Ga_0,
\end{equation}
where the symbol $\sqcup$ means a disjoint union of sets.
\end{definition}

\begin{remark}
The function $G_1\star G_2$, given by \eqref{star}, is also measurable. Moreover, the classes of functions $L_\ls^0(\Ga_0)$ and~$\Bbs$ are closed with respect to $\star$-convolution, see \cite[Remarks~3.10, 3.12]{KK2002}.
\end{remark}

\begin{remark}
The equality \eqref{star} may be rewritten in the following form
\begin{equation}\label{star-q}
 (G_1\star G_2)(\eta):=
 \sum_{\xi_1\cup\xi_2=\eta}G_1(\xi_1)\,G_2(\xi_2).
\end{equation}
In turn, the convolution \eqref{ast} may be rewritten in the form similar to
 \eqref{star-q}, namely,
\begin{equation}\label{ast-q}
 (G_1 * G_2)(\eta)=\sum_{\xi_1\sqcup\xi_2=\eta}
 G_1(\xi_1)\,G_2(\xi_2).
\end{equation}
Comparing the right hand sides of \eqref{ast-q} and~\eqref{ast-q}, it is easily seen that the sum in the definition of the $\ast$-convolution is a part of the sum in the definition of the $\star$-convolution.
\end{remark}

A function $k\in L^0(\Ga_0)$ is said to be a positive definite in the sense of the $\star$-convolution, if
\begin{equation}\label{posdefstar}
 \int_{\Ga_0} (G\star G)(\eta) k(\eta)\, d\la(\eta) \geq0
\end{equation}
for all $B\in\Bbs$. We would like to check now either the set of all positive definite functions in the sense of the $\star$-convolutions be a closed set with respect to the $*$-convolution.

We start with the following convolution between measurable functions $G_1$ and~$G_2$ on
$\Ga_0^2$:
\begin{equation}\label{star2}
 (G_1\Star G_2)(\eta^+,\eta^-):=
 \sum_{\substack{\xi^+_1\sqcup\xi^+_2\sqcup\xi^+_3=\eta^+\\
 \xi^-_1\sqcup\xi^-_2\sqcup\xi^-_3=\eta^-}}
 G_1(\xi_1^+\cup\xi_2^+,\xi_1^-\cup\xi_2^-)\,
 G_2(\xi_2^+\cup\xi_3^+,\xi_2^-\cup\xi_3^-).
\end{equation}
A measurable function $k:\Ga_0^2\to\R$ is said to be a positive definite in the sense of the $\Star$-convolution if for all $G\in\BBs$
\begin{equation} \label{ex:posdef2}
\int_{\Ga_0^2}(G\Star G)(\eta^+,\eta^-) {k}(\eta^+,\eta^-)
d\lambda(\eta^+)d\lambda(\eta^-)\geq 0.
\end{equation}

\begin{proposition}\label{critposdef}
Let functions $k_i:\Ga_0\to\R$, $i=1,2$ be measurable. Then the function $k(\eta)=(k_1\ast k_2)(\eta)$ is positive definite in the sense of the $\star$-convolution on~$\Ga_0$, if only the function $\widehat{k}(\eta^+,\eta^-):=k_1(\eta^+)k_2(\eta^-)$ is positive definite in the sense of the $\Star$-convolution on~$\Ga_0^2$.
\end{proposition}
\begin{proof}
Let $G_i\in\Bbs$ and~functions $\widetilde{G}_i\in\BBs$, $i=1,2$ are defined by analogy to~\eqref{tildeG}. Then for all $(\eta^+,\eta^-)\in\Ga^2_0$, such that~$\eta^+\cap\eta^-=\emptyset$, one has
\begin{align}
&(G_1 \star G_2)(\eta^+\cup\eta^-)=\sum_{\xi_1 \sqcup\xi_{2}\sqcup
\xi_{3}=\eta^+\cup\eta^-
} G_1(\xi _1\cup\xi _2)G_2(\xi _2\cup\xi _3)\notag\\
=&\sum_{\eta^+_1 \sqcup\eta^+_{2}\sqcup \eta^+_{3}=\eta^+}
\sum_{\eta^-_1 \sqcup\eta^-_{2}\sqcup \eta^-_{3}=\eta^-}G_1(\eta^+
_1\cup\eta^+ _2\cup  \eta^- _1\cup\eta^- _2)G_2(\eta^+ _2\cup\eta^+
_3\cup \eta^- _2\cup\eta^- _3)\notag\\=&\,(\widetilde{G}_1\Star
\widetilde{G}_2)(\eta^+,\eta^- ).\label{ex::2}
\end{align}

We proceed to show now that, for an arbitrary $z>0$,
\begin{equation}\label{eqoi}
    (\la_z\otimes\la_z) \Bigl(\bigl\{
    (\eta^+,\eta^-)\in\Ga_{0}^2\bigm| \eta^+\cap\eta^-\neq\emptyset
    \bigr\}\Bigr)=0.
\end{equation}
Indeed, for any $\eta^+\in\Ga^+_{0}$, one can define
$A_{\eta^+}:=\bigl\{ \eta^-\in\Ga^-_{0}\bigm|
\eta^+\cap\eta^-\neq\emptyset \bigr\}$.
Then we have an estimate
\begin{equation}\label{eqoii}
    \la_z (A_{\eta^+})\leq \sum_{x\in\eta^+}\la_z\Bigl(\bigl\{
\eta^-\in\Ga^-_{0}\bigm| x\in\eta^-\bigr\}\Bigr) =0,
\end{equation}
where we used \eqref{zeroset1}.
Next, using
\[
(\la_z\otimes\la_z) \Bigl(\bigl\{
(\eta^+,\eta^-)\in\Ga_{0}^2\bigm| \eta^+\cap\eta^-\neq\emptyset
\bigr\}\Bigr)=\int_{\Ga^+_{0}}\la_z\bigl(A_{\eta^+}\bigr)d\la_z(\eta^+),
\]
one has that \eqref{eqoii} implies \eqref{eqoi}.

Then, for any $G\in\Bbs$ and~$\widetilde{G}\in
\BBs$, given by~\eqref{tildeG}, we derive from \eqref{minlosid-ast} that
\begin{align}
&\int_{\Ga_0} (G\star G)(\eta)k(\eta)d\la(\eta)= \int_{\Ga_0}
(G\star G)(\eta)(k_1\ast k_2)(\eta)d\la(\eta)\notag\\ = &
\int_{\Ga_0^2} (G\star G)(\eta^+\cup\eta^-)k_1(\eta^+)
k_2(\eta^-)d\la(\eta^+)d\la(\eta^-)\notag\\
=&\int_{\Ga_0^2}
(\widetilde{G}\Star\widetilde{G})(\eta^+,\eta^-)k_1(\eta^+)
k_2(\eta^-)d\la(\eta^+)d\la(\eta^-),\label{ex::3}
\end{align}
where we used \eqref{ex::2} and~\eqref{eqoi}.

The equality \eqref{ex::3} implies immediately that the positive definiteness of $\widehat{k}=k_1\otimes k_2$ in the sense of the $\Star$-convolution yields the positive definiteness of $k=k_1\ast k_2$ in the sense of the $\star$-convolution.
The statement is proved.
\end{proof}

\begin{remark}
  One can change in \eqref{posdefstar} the measure $k\,d\la$ onto any measure $\rho\in\Mlf$ and define by an analogy the notion of a measure on $\Ga_0$ which is positive definite with respect to the $\star$-convolution. Then the results of Proposition~\ref{critposdef} may be reformulated for measures $\rho_1, \rho_2$, if we only know that
  $(\rho_1\otimes\rho_2) \bigl(\bigl\{(\eta^+,\eta^-)\in\Ga_{0}^2\bigm| \eta^+\cap\eta^-\neq\emptyset\bigr\}\bigr)=0$.
\end{remark}

\def\cprime{$'$}

\end{document}